\documentclass{amsart}

\usepackage{geometry}
\usepackage{physics}
\usepackage{times}
\usepackage{longtable}

\usepackage{enumerate}
\usepackage[all]{xy} % diagrams

\usepackage{amsmath, amssymb, amsfonts, latexsym, mdwlist, amsthm,amscd}
\usepackage{caption, subcaption}

\usepackage{hologo} % for TeX engine logos
\usepackage{booktabs} % for nice tables
\usepackage{tablefootnote} % for table footnotes
\usepackage{tikz}
\usetikzlibrary{calc,trees,arrows,chains,shapes.geometric,%
	decorations.pathreplacing,decorations.pathmorphing,shapes,%
	matrix,shapes.symbols}

\tikzset{
	>=stealth',
	punktchain/.style={
		rectangle,
		rounded corners,
		draw=black, thick,
		minimum height=3em,
		text centered,
		on chain},
	line/.style={draw, thick, <-},
	element/.style={
		tape,
		top color=white,
		bottom color=blue!50!black!60!,
		minimum width=8em,
		draw=blue!40!black!90, very thick,
		text width=10em,
		minimum height=3.5em,
		text centered,
		on chain},
	every join/.style={->, thick,shorten >=1pt},
	decoration={brace},
	tuborg/.style={decorate},
	tubnode/.style={midway, right=2pt},
}
\usepackage{cite}
\usepackage{url}

\usepackage{verbatim}
\usepackage{mathtools}

\usepackage{pgf,tikz}
\usepackage{tikz-cd}
\usetikzlibrary{calc,matrix,patterns,shadings,backgrounds,fit}
\pgfdeclarelayer{myback}
\pgfsetlayers{myback,background,main}
\usepackage{stmaryrd}
\usepackage{diagbox}
\usepackage{extarrows}

\usepackage[bookmarks, colorlinks, breaklinks, pdftitle={Stab on X 1 Feb},
pdfauthor={Chunyi Li}]{hyperref}
\hypersetup{linkcolor=blue,citecolor=blue,filecolor=black,urlcolor=blue}

\usepackage{paralist}
\setdefaultenum{(a)}{(i)}{}{}

\usepackage{pgfplots}
\pgfplotsset{compat=1.18}
\usepackage{mathrsfs}
\usetikzlibrary{arrows}

\numberwithin{equation}{section}

\def\C{\ensuremath{\mathbb{C}}}

\def\Q{\ensuremath{\mathbb{Q}}}
\def\R{\ensuremath{\mathbb{R}}}
\def\Z{\ensuremath{\mathbb{Z}}}

\def\ch{\mathop{\mathrm{ch}}\nolimits}

\def\Coh{\mathop{\mathrm{Coh}}\nolimits}

\def\Gr{\mathop{\mathrm{Gr}}\nolimits}

\def\Hom{\mathop{\mathrm{Hom}}\nolimits}
\def\id{\mathop{\mathrm{id}}\nolimits}

\def\Kn{\mathop{\mathrm{K}_{\mathrm{num}}}}

\def\min{\mathop{\mathrm{min}}\nolimits}

\def\Pic{\mathop{\mathrm{Pic}}}

\def\rk{\mathop{\mathrm{rk}}}

\def\Stab{\mathop{\mathrm{Stab}}\nolimits}

\def\blank{\underline{\hphantom{A}}}

\def\Db{\mathrm{D}^{b}}

\def\abs#1{\left\lvert#1\right\rvert}

\def\pb#1{#1^\sharp}
\def\pf#1{#1_\sharp}

\makeatletter
\newtheorem*{rep@theorem}{\rep@title}
\newcommand{\newreptheorem}[2]{%
\newenvironment{rep#1}[1]{%
 \def\rep@title{#2 \ref{##1}}%
 \begin{rep@theorem}}%
 {\end{rep@theorem}}}
\makeatother

\theoremstyle{definition}
\newtheorem{Thm}{Theorem}[section]
\newreptheorem{Thm}{Theorem}
\newtheorem{Prop}[Thm]{Proposition}

\newtheorem{Lem}[Thm]{Lemma}

\newtheorem{Cor}[Thm]{Corollary}
\newreptheorem{Cor}{Corollary}

\newtheorem{thm-int}{Theorem}

\newtheorem{Def-s}[Thm]{Definition}
\newtheorem{Def}[Thm]{Definition}
\newtheorem{Rem}[Thm]{Remark}

\newtheorem{Not}[Thm]{Notation}

\def\C{\ensuremath{\mathbb{C}}}

\def\Q{\ensuremath{\mathbb{Q}}}
\def\R{\ensuremath{\mathbb{R}}}
\def\Z{\ensuremath{\mathbb{Z}}}

\def\cD{\ensuremath{\mathcal D}}
\def\cE{\ensuremath{\mathcal E}}
\def\cF{\ensuremath{\mathcal F}}

\def\cL{\ensuremath{\mathcal L}}

\def\cO{\ensuremath{\mathcal O}}
\def\cP{\ensuremath{\mathcal P}}

\def\cT{\ensuremath{\mathcal T}}

\def\bP{\ensuremath{\mathbf P}}

\def\RRR{\mathfrak R}
\def\SSS{\mathfrak S}

\def\iff{\; \Longleftrightarrow \;}

\usepackage{todonotes}

\def\Cone{\mathrm{Cone}}

 \def\wbn{(\Z/2\Z)^n\rtimes S_n}

\def\Forg{\ensuremath{\mathrm{Forg}}}

\def\lsm{\ensuremath{\lesssim}}
\def\lsmeq{\ensuremath{\lessapprox}}

\def\Grg{\mathrm{K}_0}

\def\inv{\ensuremath{\mathrm{inv}}}

\title{A remark on stability conditions on smooth projective varieties}
\author{Chunyi Li}
\address{C. L.:
Mathematics Institute, University of Warwick,
Coventry, CV4 7AL,
United Kingdom}
\email{C.Li.25@warwick.ac.uk}
\urladdr{https://sites.google.com/site/chunyili0401/}

\begin{document}
\begin{abstract}
    Let $X$ be a smooth projective variety over $\mathbb C$. In this paper, we prove that $\mathrm{D}^b(X)$, the bounded derived category of coherent sheaves on $X$, always admits stability conditions in the sense of Bridgeland.
\end{abstract}
\maketitle

\setcounter{tocdepth}{1}
\tableofcontents

\section{Introduction}
The notion of a stability condition on a triangulated category $\cT$ is introduced in \cite{Bridgeland:Stab}.  One main result of \cite{Bridgeland:Stab} is that the space $\Stab(\cT)$ of all stability conditions, whenever non-empty, on $\cT$ naturally forms a complex manifold. In this paper, we prove the non-emptiness of stability conditions on $\Db(X)$, the bounded derived category of coherent sheaves on a smooth projective variety $X$.
\begin{Thm}\label{thm:main}
     Let $X$ be a smooth projective variety over $\mathbb C$. Then there exists a stability condition on $\Db(X)$.  
\end{Thm}

\subsection*{Bridgeland stability conditions}
We briefly recall some notions of stability conditions.
\begin{Def}[\!{\cite[Definition 3.3]{Bridgeland:Stab}}]\label{def:slicing}
Let $\cT$ be a $\C$-linear triangulated category. A \emph{slicing} $\cP$ on $\cT$ is a collection of full additive subcategories $\cP(\theta)\subset \cT$  indexed by $\theta\in \R$, satisfying the following conditions:
\begin{enumerate}
    \item For any $\theta\in \R$, we have $\cP(\theta)[1]=\cP(\theta+1)$.
    \item If $\theta_1>\theta_2$ and $F_i\in \operatorname{Obj}(\cP(\theta_i))$ for $i=1,2$, then $\Hom(F_1,F_2)=0$;
    \item \label{c}Every non-zero object $E \in \cT$ admits a finite sequence of distinguished triangles 
    \begin{equation*}
        \begin{tikzcd}[column sep=tiny]
0=E_0  \arrow[rr]& & E_1 \arrow[dl] \arrow[rr]& & E_2\arrow[r]\arrow [dl]& \cdots\arrow[r] &E_{m-1}\arrow [rr]& & E_m=E\arrow[dl]\\
& A_1 \arrow[ul,dashed, "+1"] && A_2\arrow[ul,dashed, "+1"] &&&& A_m \arrow[ul,dashed, "+1"]
\end{tikzcd}
    \end{equation*}
such that  each nonzero $A_i=\mathrm{Cone}(E_{i-1}\rightarrow E_i)$ belongs to $\cP(\theta_i)$ with real numbers $\theta_1>\dots >\theta_m$.
\end{enumerate}
\label{def:slicing}
\end{Def}

\noindent We denote by $\Grg(\cT)$ the Grothendieck group of $\cT$.

\begin{Def}[\!{\cite[Definition 1.1]{Bridgeland:Stab}}]
A \emph{Bridgeland pre-stability condition} on $\cT$ is a pair of data $\sigma=(\cP,Z)$, where 
\begin{itemize}
    \item $\cP$ is a slicing on $\cT$;
    \item $Z\colon\Grg(\cT)\rightarrow \C$ is a group homomorphism, called the \emph{central charge};
\end{itemize}
such that for any non-zero object $E\in\cP(\theta)$, we have 
\begin{align*}
    Z([E])\in \R_{>0}\cdot e^{i\pi \theta}.
\end{align*}
\label{def:prestab}
\end{Def}
It is always clear from the context when $E$ denotes an element of $\Grg(\cT)$ rather than an object of $\cT$, and we therefore omit the brackets $[-]$ in the notation.
\begin{Not}\label{rem:slicing}
A nonzero object $E \in \cP(\theta)$ is called $\sigma$\emph{-semistable} of \emph{phase} $\theta$, and we write
\[
\phi_\sigma(E) \coloneqq \theta .
\]
Simple objects in $\cP(\theta)$ are called $\sigma$\emph{-stable}.

The sequence of distinguished triangles appearing in Definition~\ref{def:slicing}(c) is unique up to isomorphism and is called the \emph{Harder--Narasimhan (HN) filtration} of the object $E$. The objects $A_i$ occurring in this filtration are called the \emph{HN factors} of $E$. We define
\[
\phi^+_\sigma(E) \coloneqq \theta_1, 
\qquad
\phi^-_\sigma(E) \coloneqq \theta_m .
\]

For an interval $I \subset \R$, we denote by $\cP(I)$ the extension-closed subcategory generated by the subcategories
$\{\cP(\theta) \mid \theta \in I\}$.
Equivalently, $\cP(I)$ is the smallest full additive subcategory of $\cT$ containing all objects whose HN factors have phases lying in $I$.
\end{Not}

Let $\Phi$ be an exact autoequivalence of $\cT$, and denote by $\Phi_*:\Grg(\cT)\to \Grg(\cT)$ the induced isomorphism on the Grothendieck group. For a pre-stability condition $\sigma=(\cP,Z)$ on $\cT$, we define the action of $\Phi$ on $\sigma$ by
\begin{align*}
    \Phi\cdot \sigma\coloneq (\Phi\circ\cP,Z\circ \Phi^{-1}_*).
\end{align*}

\noindent Let  $\Lambda$ be a free abelian group of finite rank, and let $v\colon \Grg(\cT)\twoheadrightarrow \Lambda$ be  a surjective  group homomorphism.
\begin{Def}[\!{\cite[Section 1.2]{Kontsevich-Soibelman:stability}}]
A pre-stability condition $(\cP,Z')$ is said to  satisfy the \emph{support property} with respect to $(\Lambda,v)$ if:
\begin{itemize}
\item the central charge $Z'$ factors through $\Lambda$, in other words, there exists a group homomorphism $Z\colon \Lambda \to \C$ such that $Z' = Z \circ v$;
\item 
Fix any norm $\|\blank\|$ on the finite dimensional vector space $\Lambda_\R\coloneqq \Lambda \otimes \R$. There exists a constant $C_\sigma>0$ such that, for any $\sigma$-semistable object $E$, we have
\begin{equation}\label{eq:Sp}
|Z(v(E))|\geq C_\sigma\|v(E)\|.    
\end{equation}
\end{itemize}
A pre-stability condition satisfying the support property is called a \emph{stability condition} (with respect to $(\Lambda,v)$). The set of all such stability conditions is denoted by \[\Stab_{(\Lambda,v)}(\cT).\]
\label{def:supportproperty}
\end{Def}

The set $\Stab_{(\Lambda,v)}(\cT)$  carries a natural topology induced by the following generalized metric. For stability conditions $\sigma_1=(\cP_1,Z_1)$ and $\sigma_2=(\cP_2,Z_2)$, their distance is defined by
\[
\mathrm{dist}(\sigma_1,\sigma_2)\coloneqq\sup_{\substack{ 0 \neq E\in\cT}}\left\{\abs{\phi_{\sigma_1}^-(E)-\phi_{\sigma_2}^-(E)},\abs{\phi_{\sigma_1}^+(E)-\phi_{\sigma_2}^+(E)}, \left\|Z_1-Z_2\right\|\right\}\in [0,+\infty].
\]
%where $\|\blank\|$ denotes any fixed norm on the finite-dimensional vector space $\mathrm{Hom}_{\mathbb Z}(\Lambda, \mathbb C)$.

\begin{Thm}[Bridgeland's Deformation Theorem,\cite{Bridgeland:Stab,Arend:shortproof}]\label{thm:Bridgeland}
The map forgetting the slicing
\[
\Forg_Z \colon\Stab_{(\Lambda,v)}(\cT) \longrightarrow \mathrm{Hom}_{\mathbb Z}(\Lambda, \mathbb C)\cong \C^{\rk(\Lambda)}, \qquad \sigma =(Z,\cP) \longmapsto Z 
\]
is a local homeomorphism with respect to the distance topology. In particular, whenever nonempty, the space $\Stab_{(\Lambda,v)}(\cT)$ is naturally a complex manifold of dimension $\rank(\Lambda)$.
\end{Thm}

\smallskip

\noindent\textbf{Convention.} When $\mathcal T = \Db(X)$, the bounded derived category of coherent sheaves on $X$, we write $\Stab(X)$ for $\Stab(\Db(X))$, and always speak of stability conditions on $X$ rather than on $\Db(X)$.

\medskip

\noindent\textbf{Proof strategy}: Our proof of Theorem \ref{thm:main} makes essential use of the following results and techniques:
\begin{itemize}
\item \cite{Polishchuk:families-of-t-structures} For a finite morphism $f:Y\to X$ between two smooth projective varieties, under suitable technical assumption on $\sigma_Y$ (resp. $\sigma_X$), one can define an induced stability condition $\pf f\sigma_Y$ (resp. $\pb f\sigma_X$) on $X$ (resp. $Y$).

    \item \cite{yucheng:stabprodvar} For every $(a,b)\in\Q_{>0}\times \Q$, there exists a stability condition $\sigma^{a,b}$ on the product $E^n$ of an elliptic curve.
\item \cite{FLZ:ab3,stabhk} The stability conditions $\sigma^{a,b}$ are $\wbn$-invariant. Moreover,  they satisfy the support property with respect to a large lattice.

\item \cite{realred} We introduce the following properties of stability conditions, see Definition \ref{def:lsm}:
\begin{align*}
    &\text{Bayer property} & \sigma\lsmeq\sigma\otimes \cO(H);\\
    &\text{Restriction-$n$ property}& \sigma\otimes \cO(nH)\lsm \sigma[1].
\end{align*}
Here $H$ is an ample divisor of the variety. Intuitively, Bayer property is a slightly stronger property than requiring all skyscraper sheaves to be stable, while a stability condition satisfying the Restriction-$n$ property tends toward the so-called large volume limit as $n \to \infty$. Both properties behave well under the operations $\pf f$ and $\pb f$.
\end{itemize}

%Modulo out these results, the whole argument is simple:

The proof proceeds in four main steps:
\begin{enumerate}[Step 1.]
    \item Lemma \ref{lem:bayerE}: the stability condition $\sigma^{a,b}$ from \cite{yucheng:stabprodvar} always satisfies Bayer property.
    \item Proposition \ref{prop:resnforEn}: for any $N\geq 1$, when $a$ is sufficiently large, the stability condition $\sigma^{a,b}$ also satisfies the Restriction-$N$  property.
    \item Theorem \ref{thm:EntoPn}: Let $\pi:E^n\to E^n/\wbn\cong \bP^n$ be the quotient morphism. Then $\pf\pi\sigma^{a,b}$ is a well-defined stability condition on $\bP^n$ that inherits the Bayer property. 
    \item Theorem \ref{thm:stabonall}: For any closed embedding $\iota:X\hookrightarrow \bP^n$, there exist an integer $N\geq 1$ such that $\pb\iota\sigma$ is a stability condition,  provided $\sigma$ on $\bP^n$ satisfies the Bayer property and Restriction-$N$ property.
\end{enumerate}

Step 1 uses a trick from \cite{FLZ:ab3} involving Pic$^0(E)$-invariance of the stability condition. Step 2 builds on two results from \cite{stabhk}: the support property of $\sigma^{a,b}$ with respect to a large lattice, and the injectivity of the forgetting map $\operatorname{Forg}_Z$. The conclusion follows by pulling back $\sigma^{a,b}$ via an isogeny of sufficiently high degree.

Step 3 is the most technically involved part. Using Proposition \ref{prop:pfofstab} and the result from Step 1, the problem is reduced to expressing the object $\pi^* \pi_*(F)$ as a filtration with factors  $w^* (F) \otimes \mathcal{L}^{-1}$ for some translations $w$ and effective line bundles $\mathcal{L}$. 

To achieve this, we factor $\pi$ as the composition $E^n\to (\bP^1)^n\to \bP^n$. The first morphism, $E^n \to (\bP^1)^n$, is further decomposed into a sequence of double covers, handled in Lemma \ref{lem:doublecoverind}. For the second morphism, $(\bP^1)^n \to \bP^n$, we construct a filtration of the structure sheaf $\mathcal{O}_{(\bP^1)^n \times_{\bP^n} (\bP^1)^n}$, viewed as a coherent sheaf on $(\bP^1)^n \times (\bP^1)^n$. This setup is detailed in Lemma \ref{lem:Gammafiltration}. Each successive quotient in this filtration is isomorphic to the structure sheaf $\mathcal{O}_{\Gr_w}$ of the graph of an element $w\in S_n$, twisted by a line bundle $\mathcal{L}_w^{-1}$.

Step 4  has been treated in a more general context in \cite[Section 6]{realred}.  The initial proof, as in the arXiv version, contained a mistake. The corrected version, which is available on the author’s website, is also included here for completeness and consistency.

\subsection*{Acknowledgements}
The author would like to thank Arend Bayer, Aaron Bertram, Yiran Cheng, Naoki Koseki, Zhirui Li, Peize Liu, Yuan Lu, Emanuele Macrì, Alex Perry, Laura Pertusi, Yu Qiu, Lei Song, Paolo Stellari, Xiao Wang, Chunkai Xu, Qizheng Yin, and Xiaolei Zhao for enlightening discussions and helpful comments. The author is grateful to Emanuele Macrì, Alex Perry, and Xiaolei Zhao for pointing out a mistake in \cite{realred}.

\noindent The author is supported by the Royal Society URF/R1/201129 and 251025, \emph{Stability Conditions and Applications in Algebraic Geometry}. The author expresses his deepest gratitude to his advisor, Tom Nevins, for the long-term support and encouragement, and for teaching him the theory of symmetric polynomials, which remains an enduring legacy.

\section{Bayer Property}
The following relation between stability conditions, introduced in \cite{realred}, will be useful in our arguments.
\begin{Def}\label{def:lsm}
For two pre-stability conditions $\sigma,\omega$ (or equivalently, their associated slicings) on $\cT$, we define
\begin{align*}
\sigma\lsm \omega & :\iff \cP_{\sigma}(\theta)\subset\cP_{\omega}(<\theta) \text{ for every }\theta\in \R.\\
\sigma\lsmeq \omega & :\iff \cP_{\sigma}(\theta)\subset\cP_{\omega}(\leq\theta) \text{ for every }\theta\in \R.
\end{align*}
\end{Def}

The relation is transitive, in other words, $\sigma_1\lsmeq \sigma_2$ and $\sigma_2\lsmeq\sigma_3$ imply $\sigma_1\lsmeq \sigma_3$.

%The following observations will be useful.
\begin{Lem}[{\cite[Lemma 4.11]{realred}}]\label{lem:lsm1}
Let $\sigma,\omega\in\Stab(\cT)$ and $\Phi$ be an exact autoequivalence on $\cT$. Then the followings are equivalent:
\begin{enumerate}[(1)]
    \item $\sigma\lsmeq \omega$.
    %\item For every non-zero $E\in \cT$, $\phi^+_\sigma(E)\geq \phi^+_\omega(E)$ and $\phi^-_\sigma(E)\geq\phi^-_\omega(E)$.
    \item For every $\sigma$-stable object $E\in \cT$, $\phi^+_\omega(E)\leq \phi_\sigma(E)$. 
    \item For every $\omega$-stable object $E\in\cT$, $\phi_\omega(E)\leq\phi^-_\sigma(E)$.
   % \item $\cP_\omega(\theta)\subset \cP_\sigma(\geq\theta)$ for every $\theta\in \R$.
    \item $\Phi(\sigma)\lsmeq\Phi(\omega)$.
\end{enumerate}
All statements  hold for $\lsm$ by replacing $\leq$ with $<$. 
\end{Lem}

\begin{Lem}\label{lem:bayerE}
Let  $\sigma$ be a stability condition on $C\times M$, where $C$ is a smooth projective curve with  genus $\geq 1$. For any closed point $q\in C$ and  $H=\{q\}\times M$, we have 
\begin{align*}
    \sigma\lsmeq \sigma\otimes\cO(H).
\end{align*}
\end{Lem}
\begin{proof}
    Let 
    \begin{align*}
        \iota:H\to C\times M
    \end{align*} denote the inclusion. For any   $\sigma$-stable object $F$, we have  the distinguished triangle:
    \begin{align}\label{eq33}
        \iota_*\iota^*F[-1]\to F\otimes \cO(-H)\to F\to \iota_*\iota^*F.
    \end{align}
   \noindent\textbf{Claim}: $\phi^+_\sigma(\iota_*\iota^*F)\leq \phi_\sigma(F)+1$.
    \begin{proof}[Proof of the Claim]
    Let $G$ be the first HN factor of $\iota_*\iota^*F$ with respect to $\sigma$, so that
    \[\phi^+_\sigma(\iota_*\iota^*F)=\phi_\sigma(G).\] By \cite[Theorem 3.5.1]{Polishchuk:families-of-t-structures} (see also \cite[Remark 2.24]{stabhk} for more details), the group $\Pic^0(C)$ acts trivially on $\sigma$. As $\Pic^0(C)$ acts trivially on $\iota_*\iota^*F$, every HN factor of $\iota_*\iota^*F$ is also fixed by $\Pic^0(C)$. In particular, the object $G$ is supported on $H$ as well.

    Applying $\Hom(G,-)$ to the distinguished triangle \eqref{eq33}, we obtain a long exact sequence
    \begin{align}
        \dots \to \Hom(G,F)\to \Hom(G,\iota_*\iota^*F)\to \Hom(G, F\otimes \cO(-H)[1])\to \dots.
    \end{align}
    As $G$ is supported on $H$, we have $\Hom(G,F\otimes \cO(-H)[1])=\Hom(G,F[1])$. 
    
    Since $\Hom(G,\iota_*\iota^*F)\neq 0$, we have $\Hom(G,F)$ or $\Hom(G,F[1])\neq 0$.
    
   As both $G$ and $F$ are $\sigma$-semistable, this implies $\phi_\sigma(G)\leq \phi_\sigma(F[1])=\phi_\sigma(F)+1$, proving the claim.  
    \end{proof}
   Returning to the proof of the lemma, we have  
    \begin{align}\label{eq:244}
       \phi^+_{\sigma\otimes \cO(H)}(F)= \phi^+_\sigma(F\otimes \cO(-H))\leq \max\{\phi_\sigma(F),\phi^+_\sigma(\iota_*\iota^*F[-1])\}= \phi_\sigma(F).
    \end{align}
   Here, the `$\leq$' is due to the distinguished triangle \eqref{eq33}, while the last `$=$' is due to the Claim.
   
   Since \eqref{eq:244} holds for every $\sigma$-stable object $F$, Lemma \ref{lem:lsm1} implies that $\sigma\lsmeq\sigma\otimes \cO(H)$ as claimed.
\end{proof}

\section{Inducing Stability conditions}
\begin{Def}\label{def:polyshchukstabs}
Let $f\colon Y \to X$ be a finite morphism between smooth projective varieties.

Given a pre-stability condition $\sigma_Y=(\cP,Z)$ on $\Db(Y)$, we define its \emph{pushforward} to $\Db(X)$ as 
\[
\pf f\sigma_Y \coloneqq (\pf f\cP,\pf f Z),
\]
where
\begin{align*}
   \pf f\cP(\theta) 
   &\coloneqq \{\, E\in \Db(X)\mid f^*E\in \cP(\theta)\,\},
   \qquad \text{for every }\theta\in \R, \\
   \pf f Z 
   &\coloneqq Z\circ f^*\colon 
   \Grg(X)\xrightarrow{\,f^*\,}\Grg(Y)\to \C .
\end{align*}

Conversely, given a pre-stability condition $\sigma_X=(\cP,Z)$ on $\Db(X)$, we define its \emph{pullback} to $\Db(Y)$ as 
\[
\pb f\sigma_X \coloneqq (\pb f\cP,\pb f Z),
\]
where
\begin{align*}
   \pb f\cP(\theta) 
   &\coloneqq \{\, E\in \Db(Y)\mid f_*E\in \cP(\theta)\,\},
   \qquad \text{for every }\theta\in \R, \\
   \pb f Z 
   &\coloneqq Z\circ f_*\colon 
   \Grg(Y)\xrightarrow{\,f_*\,}\Grg(X)\to \C .
\end{align*}
\end{Def}

\begin{Lem}\label{lem:tensorpf}
Assume that $\pf f\sigma$ is a pre-stability condition on $\Db(X)$. Then for any line bundle $\cL$ on $X$, we have
\[
(\pf f\sigma)\otimes \cL \;=\; \pf f(\sigma\otimes f^*\cL).
\]
\end{Lem}

\begin{proof}
The statement follows directly from the definitions.

\noindent\emph{Slicings.}
For any $\theta\in\R$, we compute
\begin{align*}
\cP_{(\pf f\sigma)\otimes \cL}(\theta)
&= \cP_{\pf f\sigma}(\theta)\otimes \cL 
= \{\,F \mid f^*F\in \cP_\sigma(\theta)\,\}\otimes \cL 
= \{\,F\otimes \cL \mid f^*F\in \cP_\sigma(\theta)\,\} \\
&= \{\,G \mid f^*(G\otimes \cL^{-1})\in \cP_\sigma(\theta)\,\} 
= \{\,G \mid f^*G\in \cP_\sigma(\theta)\otimes f^*\cL\,\} 
= \{\,G \mid f^*G\in \cP_{\sigma\otimes f^*\cL}(\theta)\,\} \\
&= \cP_{\pf f(\sigma\otimes f^*\cL)}(\theta).
\end{align*}

\noindent\emph{Central charges.}
For any $F\in\Db(X)$, we have
\begin{align*}
Z_{(\pf f\sigma)\otimes \cL}(F)
= Z_\sigma\!\bigl(f^*(F\otimes \cL^{-1})\bigr) 
= Z_{\sigma\otimes f^*\cL}(f^*F) 
= Z_{\pf f(\sigma\otimes f^*\cL)}(F).
\end{align*}
This completes the proof.
\end{proof}

\begin{Lem}\label{lem:pbpfcompare}
Let $\sigma \lsmeq \omega$ be pre-stability conditions. Assume that both $\pf f\sigma$ and $\pf f\omega$ (respectively, $\pb f\sigma$ and $\pb f\omega$) are pre-stability conditions. Then
\[
\pf f\sigma \lsmeq \pf f\omega
\qquad
(\text{respectively, } \pb f\sigma \lsmeq \pb f\omega).
\]
\end{Lem}

\begin{proof}
We prove the statement for $\pf f$; the argument for $\pb f$ is identical.

Fix any $\theta\in\R$ and  $F\in \cP_{\pf f\sigma}(\theta)$. By definition, we have $f^*F \in \cP_\sigma(\theta)$.

Let
\begin{align}\label{eq:filF}
0 = F_0 \to F_1 \to \cdots \to F_s = F   
\end{align}

be the HN filtration of $F$ with respect to $\pf f\omega$, and denote the corresponding HN factors by
\[
G_j = \Cone(F_{j-1}\to F_j)\in \cP_{\pf f\omega}(\theta_j),
\qquad
\theta_1 > \cdots > \theta_s.
\]
By definition of $\pf f\omega$, we have
\[
f^*G_j \in \cP_\omega(\theta_j)
\quad \text{for all } j.
\]

Pulling back the filtration \eqref{eq:filF}  via $f^*$, we obtain a filtration of $f^*F$ whose factors are
\[
f^*G_1,\, f^*G_2,\, \dots,\, f^*G_s.
\]
In particular, these are precisely the HN factors of $f^*F$ with respect to $\omega$.

By Lemma~\ref{lem:lsm1} and the assumption $\sigma\lsmeq\omega$, we obtain
\begin{align*}
\theta
= \phi_\sigma(f^*F)
\;\ge\;
\phi^+_\omega(f^*F)
= \phi_\omega(f^*G_1)
= \theta_1.
\end{align*}
Therefore, by \eqref{eq:filF}, we have
$F \in \cP_{\pf f\omega}(\le \theta)$. Note that this is for any $F\in \cP_{\pf f\sigma}(\theta)$, hence we have
\[
\cP_{\pf f\sigma}(\theta) \subset \cP_{\pf f\omega}(\le \theta) \text{ for any }\theta\in \R.
\]
By Definition \ref{def:lsm}, this is $\pf f\sigma \lsmeq \pf f\omega$.
\end{proof}

\begin{Prop}[\!{\!\cite[Section~2]{Polishchuk:families-of-t-structures}}]\label{prop:pfofstab}
Let $f\colon Y \to X$ be a surjective finite morphism between smooth projective varieties, and let $\sigma$ be a stability condition on $\Db(X)$. Assume that
\begin{equation}\label{eq23}
    f^*f_*\cP_\sigma(\theta)\subset \cP_\sigma(\le \theta),
    \qquad \forall\,\theta\in\R.
\end{equation}
Then $\pf f\sigma$ is a stability condition on $\Db(X)$.
\end{Prop}

\begin{proof}
We apply \cite[Theorem~2.1.2]{Polishchuk:families-of-t-structures}\footnote{It is worth noting that, in the proof of \cite[Lemma~2.1.1 and Theorem~2.1.2]{Polishchuk:families-of-t-structures}, the application of \cite[Theorem~A.1]{AJS:tstructure} requires the generating set of the pre-aisle to consist entirely of compact objects. In our paper, all varieties are assumed to be smooth, so this condition is always satisfied.} with
\[
\cD_1=\Db(X), \qquad 
\cD_2=\Db(Y), \qquad 
\Phi=f^*, \qquad 
\Psi \text{ the left adjoint of } f^*.
\]

For every $\theta\in\R$, we claim that
\begin{equation}\label{eq:key}
    f^*\circ\Psi\bigl(\cP_\sigma(\theta)\bigr)\subset \cP_\sigma(\ge \theta).
\end{equation}
Indeed, let $E\in \cP_\sigma(\theta)$ and $F\in \cP_\sigma(<\theta)$. Then
\begin{align*}
\Hom_{\Db(X)}\bigl(f^*\Psi(E),F\bigr)
\cong \Hom_{\Db(Y)}\bigl(\Psi(E),f_*F\bigr) 
\cong \Hom_{\Db(X)}\bigl(E,f^*f_*F\bigr) 
= 0,
\end{align*}
where the last equality follows from assumption~\eqref{eq23}.
This proves~\eqref{eq:key}.

By the same argument as in
\cite[Proposition~2.2.1 and Corollary~2.2.2]{Polishchuk:families-of-t-structures},
it follows that the pair $\pf f\sigma$ defines a pre-stability condition on $\Db(X)$.

Assume that $\sigma$ satisfies the support property with respect to a lattice
$v\colon \Grg(Y)\to \Lambda$.
It is clear from the definition that
$\pf f\sigma$ satisfies the support property with respect to
$
v\circ f^*\colon \Grg(X)\to \Lambda',
$
where $\Lambda'\subset \Lambda$ is the image of $v\circ f^*$.
Hence $\pf f\sigma$ is a stability condition on $\Db(X)$.
\end{proof}

\section{Double Schubert polynomials}

Let $R=\C[x_1,x_2,\dots,x_n]$. The symmetric group $S_n$ acts on $R$ by permuting the variables. In this section, we study a natural filtration on the $R\otimes R$-module (equivalently, the $R$-bimodule)
\[
R\otimes_{R^{S_n}} R.
\]
This filtration is often called the Soergel bimodule filtration, and is discussed in much greater depth in \cite{Soergel:bimodule}. Since this filtration will play an important role later, we recall the relevant details here, following the classical exposition in \cite{Macdonald1991}.

For each $w\in S_n$, define
\begin{align*}
    \inv(w) &\coloneq \{(i,j)\mid 1\le i<j\le n,\; w(i)>w(j)\},\\
    \Delta_w &\coloneq \prod_{(i,j)\in \inv(w)}(x_i-x_j).
\end{align*}

We view $S_n$ as a Coxeter group with simple reflections
\[
s_j=(j\ j+1), \qquad 1\le j\le n-1,
\]
and denote by $\ell(w)$ the Bruhat length of $w$. For each $1\le j\le n-1$, define the Demazure operator
\[
\partial_j \coloneq \frac{\id - s_j}{x_j-x_{j+1}}.
\]
Given a reduced expression $w=s_{a_1}\cdots s_{a_{\ell(w)}}$, we have
\[
\partial_w \coloneq \partial_{a_1}\cdots \partial_{a_{\ell(w)}}.
\]
This operator is independent of the choice of reduced expressions. For the identity element $e\in S_n$, we set $\partial_e=\id$.

The \emph{Schubert polynomial} associated to $w\in S_n$ is defined by
\[
\SSS_w(x)\coloneq \partial_{w^{-1}w_0}\bigl(x_1^{n-1}x_2^{n-2}\cdots x_{n-1}\bigr),
\]
where $w_0$ denotes the longest element of $S_n$. It is a classical result that
$\{\SSS_w(x)\}_{w\in S_n}$
forms a free basis of $R$ as an $R^{S_n}$-module.

The \emph{double Schubert polynomial} (see, for example, \cite[(6.1) and (6.3)]{Macdonald1991}) is defined by
\begin{equation}\label{eq:doubleschub}
    \SSS_w(x;y)
    \coloneq \partial_{w^{-1}w_0}\Delta(x;y)
    = \sum_{\substack{u,v\in S_n; \; w=v^{-1}u; \\ \ell(w)=\ell(v)+\ell(u)}}
    \SSS_u(x)\,\SSS_v(-y),
\end{equation}
where
\[
\Delta(x;y)\coloneq \prod_{i+j\le n}(x_i-y_j).
\]

For any $w'\in S_n$ with $\ell(w')\leq \ell(w)$, we have
\begin{equation}\label{eq:wactionondoubleschub}
    \SSS_w(w'x;x)
    = (-1)^{\ell(w)}\,\delta_{w,w'}\,\Delta_w(x).
\end{equation}

\noindent\emph{Proof of \eqref{eq:wactionondoubleschub}.}
For any $\tau\in S_n$ and $h(x)\in R$, one has
\[
h(\tau x)
=
\Bigl(
(-1)^{\ell(\tau)}\Delta_\tau(x)\,\partial_\tau
+\sum_{\ell(\tau')<\ell(\tau)} f_{\tau,\tau'}(x)\,\partial_{\tau'}
\Bigr) h(x),
\]
for some polynomials $f_{\tau,\tau'}(x)\in R$. Applying this to
$h(x)=\partial_{w^{-1}w_0}\Delta(x;y)$ and evaluating at $y=x$, we obtain
\begin{equation}\label{eq:435}
    \SSS_w(w'x;x)
    =
    \Bigl(
    (-1)^{\ell(w')}\Delta_{w'}(x)\,\partial_{w'}
    +\sum_{\ell(\tau')<\ell(w')} f_{w',\tau'}(x)\,\partial_{\tau'}
    \Bigr)
    \partial_{w^{-1}w_0}\Delta(x;y=x).
\end{equation}

By the identity $\SSS_u(x,x)=\delta_{u,e}$ (see \cite[(6.4)]{Macdonald1991}), when $\ell(\tau')\leq \ell(w)$,
the expression
\[
\partial_{\tau'}\partial_{w^{-1}w_0}\Delta(x;y=x)
\]
is nonzero only when $\tau'=w$. Consequently, the right-hand side of
\eqref{eq:435} is nonzero only if $w'=w$, and in this case it equals
$(-1)^{\ell(w)}\Delta_w(x)$. This proves \eqref{eq:wactionondoubleschub}.
\hfill$\Box$

\medskip
Now let $R_w$ denote the $R$-bimodule with the usual left $R$-action and
right $R$-action twisted by $w$. More explicitly, the bimodule structure is
given by
\[
(R\otimes R)\times R_w \to R_w,
\qquad
(f\otimes g,\,h)\longmapsto g\bigl(x_{w(1)},\dots,x_{w(n)}\bigr)\,f\,h.
\]

There is a natural $R$-bimodule homomorphism
\[
F_w\colon R\otimes_{R^{S_n}} R \longrightarrow R_w,
\qquad
f\otimes g \longmapsto g\bigl(x_{w(1)},\dots,x_{w(n)}\bigr)\,f.
\]

There is an $R$-bimodule filtration on $R\otimes_{R^{S_n}}R$; see also
\cite{Soergel:bimodule}.

\begin{Lem}\label{lem:Gammafiltration}
The $R$-bimodule $R\otimes_{R^{S_n}}R$ admits a filtration
\[
0=\Gamma_{\binom{n}{2}+1}\subset \Gamma_{\binom{n}{2}}
\subset \cdots \subset \Gamma_1\subset \Gamma_0=R\otimes_{R^{S_n}}R
\]
by $R$-bimodules such that, for every $j$, there is an $R$-bimodule
isomorphism
\begin{equation}\label{eq:45}
\bigoplus_{\ell(w)=j} F_w\colon
\Gamma_j/\Gamma_{j+1}\xrightarrow{\;\sim\;}
\bigoplus_{\ell(w)=j} \Delta_w\cdot R_w .
\end{equation}
\end{Lem}

\begin{proof}
We use elements of $R\otimes_{R^{S_n}}R$ expressed in terms of
double Schubert polynomials as in~\eqref{eq:doubleschub} to span the filtration. For each
$w\in S_n$, define
\[
S_w \coloneq
1\otimes \SSS_w(x)
+\!\!\sum_{\substack{ \ell(u)<\ell(w); \;w=v^{-1}u;  \\ \ell(w)=\ell(v)+\ell(u)}}
\SSS_v(-x)\otimes \SSS_u(x).
\]

By \eqref{eq:doubleschub} and \eqref{eq:wactionondoubleschub}, for any
$w,w'\in S_n$ with $\ell(w')\ge \ell(w)$ we have
\begin{equation}\label{eq:47}
F_w(S_{w'})
=\SSS_{w'}(wx;x)
=(-1)^{\ell(w)}\delta_{w,w'}\,\Delta_w(x).
\end{equation}

Since $\{1\otimes \SSS_w(x)\}_{w\in S_n}$ is a free basis of
$R\otimes_{R^{S_n}}R$ as a left $R$-module, it follows that
$\{S_w\}_{w\in S_n}$ is also a free left $R$-module basis.

For each $j\in\{0,\dots,\binom{n}{2}+1\}$, define
\[
\Gamma_j \coloneq \bigoplus_{\ell(w)\ge j} R\cdot S_w,
\qquad
\Gamma_{<j} \coloneq \bigoplus_{\ell(w)< j} R\cdot S_w.
\]
Both $\Gamma_j$ and $\Gamma_{<j}$ are free left $R$-modules. Viewing
$\bigoplus_{\ell(w)<j} F_w$
as a left $R$-module homomorphism, equation~\eqref{eq:47} implies
\begin{align}
\notag\bigl(\bigoplus_{\ell(w)<j}F_w\bigr)(\Gamma_j)&=0,\\
\bigl(\bigoplus_{\ell(w)=j}F_w\bigr)(\Gamma_j/\Gamma_{j+1})
&\cong \bigoplus_{\ell(w)=j}\Delta_w\cdot R_w \label{eq:45}
\end{align}
as left $R$-modules.

By induction on $j$ and \eqref{eq:45}, the map
$\bigoplus_{\ell(w)<j}F_w$ is injective on $\Gamma_{<j}$. Consequently,
the subset $\Gamma_j$ is the kernel of the $R$-bimodule homomorphism
\[
\bigoplus_{\ell(w)<j}F_w\colon
R\otimes_{R^{S_n}}R \longrightarrow \bigoplus_{\ell(w)<j} R_w .
\]
In particular, $\Gamma_j$ is an $R$-bimodule. Hence the quotient
$\Gamma_j/\Gamma_{j+1}$ inherits an $R$-bimodule structure, and the map
\eqref{eq:45} is an $R$-bimodule isomorphism.
\end{proof}

Let $Y=(\bP^1)^n$. The symmetric group $S_n$ acts on $Y$ by permuting the
factors, and we have the quotient map
$f\colon Y \longrightarrow Y/S_n \cong \bP^n$.

For each $w\in S_n$, define
\begin{align*}
\cL_w
&\coloneq \cO_{\bP^1}(a_1)\boxtimes\cdots\boxtimes \cO_{\bP^1}(a_n),
\qquad
a_i=\text{the degree of }x_i\text{ in }\Delta_w(x), \\
\Gr_w
&\coloneq \{(w x,x)\mid x\in Y\}\subset Y\times Y .
\end{align*}

Let $\C^n\subset Y$ be the standard affine chart. The filtration of
$\cO_{\C^n\times_{\C^n/S_n}\C^n}$ as an
$\cO_{\C^n\times\C^n}$-module constructed in
Lemma~\ref{lem:Gammafiltration} extends naturally to
$\cO_{Y\times_{Y/S_n}Y}$, with appropriate degree shifts arising from
homogenization. 
\begin{comment}
\begin{Rem}
The statement that we are actually using in the proof of Proposition \ref{prop:syminduce} is a filtration for $f^*f_*E$ where $E\in \Db((\bP^1)^n)$.

More explicitly, the variety $Y$ is the $\mathbf{Proj}$ of the $\Z^n$-graded ring $\RRR=\C[p_1,q_1]\otimes \dots\otimes \C[p_n,q_n]$, $S_n$ permutes the indices of variables. The module $\RRR\otimes_{\RRR^{S_n}}\RRR$ becomes a $\Z^n$-graded $\RRR$-bimodule. For any $\Z^n$-graded  $\RRR$-module $M$, by definition, we have $f^*f_*\widetilde{M}=\widetilde{\RRR\otimes_{\RRR^{S_n}}\RRR\otimes_\RRR M}$, where $\widetilde \bullet$ stands for the sheafification. Lemma \ref{lem:Gammafiltration} shows that as a $\Z^n$-graded $\RRR$-bimodule $\RRR\otimes_{\RRR^{S_n}}\RRR$ admits a filtration factors of which are isomorphic to $\Delta_w(\frac{p_i}{q_i})\RRR_w$.
\end{Rem}
\end{comment}
This yields the following sheaf-theoretic analogue of Lemma \ref{lem:Gammafiltration}.

\begin{Cor}\label{cor:pndiagfiltration}
The sheaf $\cO_{Y\times_{Y/S_n}Y}\in \Coh(Y\times Y)$ admits a filtration
\[
0=\cE_{\binom{n}{2}+1}\subset \cE_{\binom{n}{2}}
\subset \cdots \subset \cE_1\subset \cE_0=\cO_{Y\times_{Y/S_n}Y}
\]
whose successive quotients are given by
\[
\cE_j/\cE_{j+1}
\cong
\bigoplus_{\substack{w\in S_n\; \ell(w)=j}}
\cO_{\Gr_w}\otimes p_2^*\cL_w^{-1}.
\]
\end{Cor}

\section{Stability conditions on the product of elliptic curves}
Let
\[
E^n \coloneq E_1\times \cdots \times E_n
\]
be the product of $n$ smooth elliptic curves, with $E_i \cong E$. We recall some results on stability conditions on $E^n$ from
\cite{yucheng:stabprodvar, FLZ:ab3, stabhk}.

For each $1\le i\le n$, fix a point $q_i\in E_i$ and define
\[
H_i \coloneq
E_1\times \cdots \times E_{i-1}\times \{q_i\}\times E_{i+1}\times \cdots \times E_n
\subset E^n .
\]
By abuse of notation, we also denote by $H_i$ the numerical first Chern
class of the line bundle $\cO_{E^n}(H_i)$.
Set
\[
H \coloneq H_1+\cdots+H_n .
\]

We define a lattice $(\Lambda\cong \Z^{2^n}, v)$ for $\Grg(E^n)$ by
\[
v\colon \Grg(E^n)\to \Lambda,
\qquad
[F]\longmapsto
\bigl(H_{j_1}\cdots H_{j_s}\,\ch_{n-s}(F)\bigr)_{1\le j_1<\cdots<j_s\le n}.
\]

\begin{Thm}\label{thm:Eninj}
Let $\sigma$ be a stability condition on $E^n$ with respect to the
lattice $(\Lambda,v)$. Then:
\begin{enumerate}
\item[\rm(1)]
{\cite[Theorem~1.1, Proposition~2.9]{FLZ:ab3}}
all skyscraper sheaves $\{\cO_p\}_{p\in E^n}$ are $\sigma$-stable with the same phase.
\item[\rm(2)]
{\cite[Theorem~1.1]{stabhk}}
the map
\begin{equation}\label{eq35}
\Stab_{(\Lambda,v)}(E^n)\longrightarrow
\Hom_\Z(\Lambda,\C)\times \R,
\qquad
\sigma=(\cP,Z)\longmapsto \bigl(Z,\phi_\sigma(\cO_p)\bigr)
\end{equation}
is injective.
\end{enumerate}
\end{Thm}

For every $(a,b)\in \Q_{>0}\times \Q$, define a group homomorphism
\[
Z^{a,b}\colon \Lambda\longrightarrow \C
\]
by
\begin{align}
\label{eq:zab}
Z^{a,b}(v(F))
&= -\int_{E^n} e^{-(b+ia)H}\,\ch(F) \\
&= -\ch_n(F)
+(b+ia)H\ch_{n-1}(F)
-\cdots
+(-1)^{n+1}\frac{(b+ia)^n}{n!}H^n\rk(F).\notag
\end{align}

\begin{Thm}[{\cite[Theorem 5.9]{yucheng:stabprodvar}, \cite[Theorem 4.5]{stabhk}}]\label{thm:Yucheng}
For any $(a,b)\in \Q_{>0}\times \Q$, there exists a stability condition
\[
\sigma^{a,b} = (\cP^{a,b}, Z^{a,b}) \in \Stab_{(\Lambda,v)}(E^n)
\]
with central charge $Z^{a,b}$ as in \eqref{eq:zab} and $\phi_{\sigma^{a,b}}(\cO_p)=1$.
\end{Thm}

Such a stability condition $\sigma^{a,b}$ is unique by
Theorem~\ref{thm:Eninj}.

\noindent For a positive integer $m$, denote the self-isogeny
\begin{align}
\pi_m \colon E^n \longrightarrow E^n, \qquad z \longmapsto m z.
\end{align}

\begin{Lem}\label{lem:pbofstabEn}
Let $\sigma$ be a stability condition on $E^n$. Then both
$\pb{\pi_m}\sigma$ and $\pf{(\pi_m)}\sigma$ are stability conditions on
$E^n$.
\end{Lem}

\begin{proof}
By \cite[Theorem 3.5.1]{Polishchuk:families-of-t-structures} (see also
\cite[Remark 2.24]{stabhk}), the group $E^n\times \Pic^0(E^n)$ acts
trivially on $\Stab_{(\Lambda,v)}(E^n)$. In particular, for any
$F\in \cP_\sigma(\theta)$, $\cL \in \Pic^0(E^n)$, and $g\in
E^n$, we have
\[
g^*(F\otimes \cL) \in \cP_\sigma(\theta).
\]

Hence, for any $F\in \cP_\sigma(\theta)$, we have
\[
(\pi_m)_*\pi_m^* F = \bigoplus_{\cL^m \cong \cO} F\otimes \cL \in \cP_\sigma(\theta).
\]
By \cite[Corollary 2.2.2]{Polishchuk:families-of-t-structures}, it follows
that $\pb{\pi_m}\sigma$ is a stability condition on $E^n$.

Similarly,
\[
\pi_m^* (\pi_m)_* F = \bigoplus_{\substack{g\in \text{Torsion}_m(E^n)}} g^* F \in \cP_\sigma(\theta).
\]
By Proposition~\ref{prop:pfofstab}, we conclude that $\pf{(\pi_m)}\sigma$ is
also a stability condition on $E^n$.
\end{proof}

Note that for any $F\in\Db(E^n)$, we have
\begin{align*}
H^{n-j}\,\ch_j(\pi_m^*F) &= m^{2j}\,H^{n-j}\,\ch_j(F),\qquad
H^{n-j}\,\ch_j((\pi_m)_*F) = m^{2n-2j}\,H^{n-j}\,\ch_j(F).
\end{align*}

It follows that the pullback and pushforward central charges satisfy
\begin{align*}
\pf {(\pi_m)} Z^{a,b} &= Z^{a,b}\circ \pi_m^* = m^{2n}\,Z^{a/m^2,b/m^2},\qquad
\pb \pi_m Z^{a,b} = Z^{a,b}\circ (\pi_m)_* = Z^{m^2a,m^2b}.
\end{align*}

By Theorem~\ref{thm:Eninj}, we therefore have
\begin{align*}
\pf {(\pi_m)} \sigma^{a,b} &= m^{2n}\,\sigma^{a/m^2,b/m^2},\qquad
\pb \pi_m \sigma^{a,b} = \sigma^{m^2a,m^2b},
\end{align*}
where $m^{2n}$ is only rescaling the central charge, and does not affect the slicing.

In particular, the corresponding slicings satisfy
\begin{align}\label{eq36}
\pf {(\pi_m)} \cP^{a,b} = \cP^{a/m^2,b/m^2},\qquad
\pb \pi_m \cP^{a,b} = \cP^{m^2a,m^2b}.
\end{align}
It follows that
\begin{align}\label{eq37}
\pf {(\pi_m)} \pb \pi_m \cP^{a,b} 
= \pb \pi_m \pf {(\pi_m)} \cP^{a,b} 
= \cP^{a,b}.
\end{align}

For a line bundle $\cL$ with $\ch_1(\cL)$ numerical equivalent to $H$, note that
\begin{align*}
Z^{a,b}\bigl(v(F\otimes \cL^{-1})\bigr)
= -\ch_n^{(b+ia)H}(F\otimes \cL^{-1}) 
= -\ch_n^{(b+1+ia)H}(F) 
= Z^{a,b+1}(v(F)).
\end{align*}
By Theorem~\ref{thm:Eninj}, this implies
\begin{align}\label{eq38}
\sigma^{a,b}\otimes \cL = \sigma^{a,b+1}.
\end{align}

\begin{Prop}\label{prop:resnforEn}
For any $N\in \Z_{\ge 1}$, there exists a stability condition
$\sigma^{a,b}$ on $E^n$ as in Theorem~\ref{thm:Yucheng} such that
\[
\sigma^{a,b}\otimes \cL^N \lsm \sigma^{a,b}[1].
\]
\end{Prop}

\begin{proof}
Let $\sigma^{a,b}$ be a stability condition. By Theorems~\ref{thm:Bridgeland}
and~\ref{thm:Yucheng}, there exists $\delta_0>0$ such that for any
$|\delta|\le \delta_0$, there exists a stability condition $\omega$ with
central charge $Z^{a,b+\delta}$ as in \eqref{eq:zab}. Shrinking $\delta_0$
if necessary, we may assume that
\[
\mathrm{dist}(\sigma^{a,b},\omega)<1,
\]
so that in particular,
\begin{align}\label{eq:omega-lsm}
\omega \lsm \sigma^{a,b}[1].
\end{align}

Choose $m$ sufficiently large such that $\frac{1}{m^2}\le \delta_0$. By
Theorems~\ref{thm:Eninj} and~\ref{thm:Yucheng}, the stability condition
$\omega$ with central charge $Z^{a,b+1/m^2}$ coincides with
$\sigma^{a,b+1/m^2}$. 

By \eqref{eq36}, \eqref{eq38}, and \eqref{eq:omega-lsm}, we then have
\[
\pf{(\pi_m)}\Bigl((\pb \pi_m \cP^{a,b}) \otimes \cL \Bigr)
= \cP^{a,b+1/m^2} \lsm \cP^{a,b}[1].
\]
Applying $\pb \pi_m$ to both sides, and using Lemma~\ref{lem:pbpfcompare}
together with \eqref{eq37}, we obtain
  \[
(\pb \pi_m \cP^{a,b}) \otimes \cL \lsm \pb \pi_m \cP^{a,b}[1].
\]
Equivalently, we have
\begin{align}\label{eq:sabL} 
\sigma^{m^2 a, m^2 b} \otimes \cL \lsm \sigma^{m^2 a, m^2 b}[1].
\end{align}

For any $k\in \Z_{\ge 1}$, applying $\pb{\pi_k}$ to \eqref{eq:sabL}, then by \eqref{eq36}, \eqref{eq38}, Theorem \ref{thm:Eninj} and \ref{thm:Yucheng}, we have
\[
\sigma^{(km)^2 a, (km)^2 b} \otimes \cL^{k^2} \lsm
\sigma^{(km)^2 a, (km)^2 b}[1].
\]

Finally, by Lemmas~\ref{lem:bayerE} and~\ref{lem:lsm1}, we have
$\sigma\otimes \cL^s \lsmeq \sigma\otimes \cL^t$ whenever $s\le t$.  
The statement follows.
\end{proof}

\section{Stability conditions on smooth projective varieties}
\subsection{From $E^n$ to $(\bP^1)^n$}
Let $\tau_E\colon E\to E$, $z\mapsto -z$, be the reflection of the elliptic curve $E$, and let
$\pi_E\colon E\to E/\langle \tau_E\rangle \cong \bP^1$ be the corresponding double
cover. For any smooth projective variety $M$, these extend naturally to
\[
\tau = (\tau_E,\id_M) \colon E\times M \to E\times M, \qquad
f = (\pi_E,\id_M) \colon E\times M \to \bP^1\times M.
\]

\begin{Lem}\label{lem:doublecoverind}
Assume that $\sigma$ is a $\tau$-invariant stability condition on
$E\times M$. Then $\pf f\sigma$ is a stability condition on
$\bP^1\times M$.
\end{Lem}

\begin{proof}
We apply Proposition~\ref{prop:pfofstab} with $Y = E\times M$ and
$X = \bP^1\times M$. Let $Z = Y \times_X Y$ be the fiber product, which is a
subscheme of $Y \times Y$, and denote by $p_i$ the natural projections:
\begin{equation}\label{eq:basechange}
\begin{tikzcd}
Z \subset Y \times Y \ar{r}{p_2}\ar{d}{p_1} & Y \ar{d}{f} \\
Y \ar{r}{f} & X
\end{tikzcd}
\end{equation}

Since $f$ is finite and surjective, we have
\[
f^* f_* (-) = p_{2*}\bigl(p_1^*(-) \otimes \cO_Z\bigr).
\]

Let $E^\tau \subset E$ denote the four points fixed by $\tau_E$, and set
$D = E^\tau \times M$. Define
\[
\Gr_\tau \coloneq \{(\tau y, y)\mid y\in E\times M\} \subset Y\times Y.
\]
Then the kernel $\cO_Z$ admits a filtration
\[
0 \longrightarrow \cO_{\Gr_\tau} \otimes p_2^* \cO_Y(-D) 
\longrightarrow \cO_Z 
\longrightarrow \cO_{\Delta(Y)} \longrightarrow 0
\]
in $\Coh(Y\times Y)$. Therefore, for any $G\in \cP_\sigma(\theta)$, the object
$p_{2*}\bigl(p_1^* G \otimes \cO_Z\bigr) = f^* f_* G$ fits into a distinguished
triangle
\begin{align}\label{eq34}
\tau^*G \otimes \cO_Y(-D) \longrightarrow f^* f_* G \longrightarrow
G \xrightarrow{+}.
\end{align}

By Lemma~\ref{lem:bayerE} and \ref{lem:lsm1}, we have $\sigma \lsmeq \sigma\otimes \cO_Y(D)$.  
Since $\sigma$ is $\tau$-invariant, it follows that
\[
\phi_\sigma^+(\tau^*G\otimes \cO_Y(-D))
= \phi_{\sigma \otimes \cO_Y(D)}^+(\tau^*G)
\le \phi_\sigma(\tau^*G)=\phi_\sigma(G),
\]
where the only inequality uses Lemma~\ref{lem:lsm1}.

Combining this fact with \eqref{eq34}, we obtain
\[
\phi_\sigma^+(f^* f_* G)
\le \max \bigl\{\phi_\sigma^+(\tau^*G \otimes \cO_Y(-D)), \phi_\sigma(G)\bigr\}
= \phi_\sigma(G) = \theta.
\]

Hence the condition
\[
f^* f_* \cP_\sigma(\theta) \subset \cP_\sigma(\le \theta)
\]
is satisfied. By Proposition~\ref{prop:pfofstab}, it follows that
$\pf f \sigma$ is a stability condition on $\bP^1 \times M$.
\end{proof}

\subsection{From $(\bP^1)^n$ to $\bP^n$}

We then consider the morphism
\[
f \colon (\bP^1)^n \longrightarrow (\bP^1)^n/S_n = \mathrm{Sym}^n(\bP^1) \cong \bP^n.
\]

\begin{Prop}\label{prop:syminduce}
Let $\sigma$ be a stability condition on $(\bP^1)^n$ satisfying:
\begin{itemize}
    \item $\sigma$ is $S_n$-invariant;
    \item $\sigma \lsmeq \sigma \otimes \cL$ for every effective line bundle $\cL$.
\end{itemize}
Then $\pf f \sigma$ is a stability condition on $\bP^n$.
\end{Prop}

\begin{proof}
We apply Proposition~\ref{prop:pfofstab} with $Y = (\bP^1)^n$ and $X = \bP^n$. Let
$Z = Y \times_X Y$ be the fiber product, and denote by $p_i$ the natural projections as in \eqref{eq:basechange}. Then
$f^* f_* (-) = p_{2*} \bigl(p_1^*(-) \otimes \cO_Z \bigr).$

By Corollary~\ref{cor:pndiagfiltration}, the structure sheaf $\cO_Z$ admits a filtration with factors of the form
\[
\cO_{\Gr_w} \otimes p_2^* \cL_w^{-1}, \quad w\in S_n,
\]
where each $\cL_w$ is an effective line bundle.

Hence, for any object $E \in \Db(Y)$, the object
$p_{2*}(p_1^* E \otimes \cO_Z)$ admits a filtration with factors
\begin{align}\label{eq:filtofE}
\{\, w^* E \otimes \cL_w^{-1} \,\}_{w \in S_n}.
\end{align}

For $E\in \cP_\sigma(\theta)$, the $S_n$-invariance of $\sigma$ implies
\[
w^* E \in \cP_\sigma(\theta), \quad \forall w \in S_n.
\]
Since $\sigma \lsmeq \sigma \otimes \cL$ for any effective line bundle $\cL$, Lemma~\ref{lem:lsm1} gives
\[
\phi_\sigma^+(w^* E \otimes \cL_w^{-1}) 
= \phi_{\sigma \otimes \cL_w}^+(w^* E) \le \phi_\sigma(w^* E) = \theta.
\]

It then follows from the filtration \eqref{eq:filtofE} that
\[
f^* f_* E \in \cP_\sigma(\le \theta).
\]

By Proposition~\ref{prop:pfofstab}, the pair of data $\pf f \sigma$ is therefore a stability condition on $\bP^n$.
\end{proof}

The Weyl group of type $B_n$, $\wbn$, acts naturally on $E^n$, with $n$ reflections on each elliptic curve and permutations of the factors. We have the quotient map
\[
\pi \colon E^n \longrightarrow E^n / (\wbn) \cong \bP^n.
\]

\begin{Thm} \label{thm:EntoPn}  
Any stability condition $\sigma$ as in Theorem~\ref{thm:Yucheng} induces a stability condition
$\pf \pi \sigma$ on $\bP^n$.
\end{Thm}

\begin{proof}
We factor $\pi$ as
\[
E^n \xrightarrow{h} E^n / (\Z/2\Z)^n \cong (\bP^1)^n \xrightarrow{f} (\bP^1)^n / S_n \cong \bP^n.
\]

Since the central charge and lattice of $\sigma$ are $\wbn$-invariant, Theorem~\ref{thm:Eninj} implies that $\sigma$ is $\wbn$-invariant.  

By Lemma~\ref{lem:doublecoverind}, the pair of data $\pf h\sigma$ defines a stability condition on $(\bP^1)^n$. Moreover, it is $S_n$-invariant because $\sigma$ is.  

For any line bundle $\cL_i = p_i^* \cO_{\bP^1}(1)$ on $(\bP^1)^n$, we have
\[
\pf h \sigma \lsmeq \pf h (\sigma \otimes h^* \cL_i) = (\pf h \sigma) \otimes \cL_i,
\]
where `$\lsmeq$' follows from Lemmas~\ref{lem:bayerE} and \ref{lem:pbpfcompare}, and the `$=$' follows from Lemma~\ref{lem:tensorpf}.  

By Lemma~\ref{lem:lsm1}, it then follows that
\[
\pf h \sigma \lsmeq (\pf h \sigma) \otimes \cL
\]
for any effective line bundle $\cL$ on $(\bP^1)^n$. By Proposition~\ref{prop:syminduce}, the pair of data $\pf f(\pf h \sigma)$ is a stability condition on $\bP^n$.  

Finally, since both $f$ and $h$ are finite surjective maps, we have $\pf f \circ \pf h = \pf \pi$ by definition. The resulting stability condition $\pf \pi \sigma$ is defined with respect to the full numerical Grothendieck group $\Kn(\bP^n)$.
\end{proof}

\begin{Cor}\label{cor:stabPntendstolv}
For any $N \in \Z_{\ge 1}$, there exist stability conditions $\sigma \in \Stab_{\Grg(\bP^n)}(\bP^n)$ satisfying
\begin{align}\label{eq:BayerPn}
\sigma \lsmeq \sigma \otimes \cO(1) \quad \text{and} \quad
\sigma \otimes \cO(N) \lsm \sigma[1].
\end{align}
Moreover, all skyscraper sheaves are in $\cP_\sigma(1)$.
\end{Cor}

\begin{proof}
By Proposition~\ref{prop:resnforEn}, we may choose $a$ sufficiently large so that
\[
\sigma^{a,b} \otimes \cL^{2N} \lsm \sigma^{a,b}[1].
\]
Applying Theorem~\ref{thm:EntoPn}, the pair of data
\[
\sigma := \pf \pi \sigma^{a,b}
\]
defines a stability condition on $\bP^n$.  

Using Lemmas~\ref{lem:bayerE}, \ref{lem:tensorpf}, and \ref{lem:pbpfcompare}, we have
\begin{align*}
\sigma = \pf \pi \sigma^{a,b} &\lsmeq \pf \pi (\sigma^{a,b} \otimes \cL^2)
= \pf \pi (\sigma^{a,b} \otimes \pi^* \cO(1)) = \sigma \otimes \cO(1), \\
\sigma \otimes \cO(N) &= \pf \pi (\sigma^{a,b} \otimes \cL^{2N}) \lsm \pf \pi (\sigma^{a,b}[1]) =\sigma[1].
\end{align*}

Finally, all skyscraper sheaves on $E^n$ lie in $\cP_{\sigma^{a,b}}(1)$ by Theorem \ref{thm:Yucheng}. For any skyscraper sheaf $\cO_p$ on $\bP^n$, its pullback $\pi^* \cO_p$ lies in $\cP_{\sigma^{a,b}}(1)$, so by definition of $\pf \pi$ we have $\cO_p \in \cP_\sigma(1)$.
\end{proof}

\subsection{From $\bP^n$ to subvarieties}

\begin{Thm}\label{thm:stabonall}
Let $X$ be a smooth projective variety. Then there exists a Bridgeland stability condition on $X$.
\end{Thm}

\begin{proof}
Let $\iota: X \hookrightarrow \bP^n$ be an embedding. Then $\iota_* \cO_X$ admits a resolution:
\begin{align}\label{eq:resOX}
0 \to \cF \to \cO_{\bP^n}(-a_{n-1})^{\oplus m_{n-1}} \to \dots \to 
\cO_{\bP^n}(-a_1)^{\oplus m_1} \to \cO_{\bP^n} \to \iota_* \cO_X \to 0,
\end{align}
for some integers $a_j, m_j \ge 0$ and $\cF\in\Coh(\bP^n)$.

Choose $N \gg 0$ such that $jN \ge a_j$ for all $1 \le j \le n-1$. By Corollary~\ref{cor:stabPntendstolv}, there exists a stability condition
$\sigma$ on $\bP^n$ satisfying
\[
\sigma \lsmeq \sigma \otimes \cO_{\bP^n}(1), \quad \sigma \otimes \cO_{\bP^n}(N) \lsm \sigma[1].
\]
Then, by Lemma~\ref{lem:lsm1}, for each $1 \le j \le n-1$ we have
\begin{align}\label{eq:twistineq}
\sigma \otimes \cO_{\bP^n}(a_j) \lsmeq \sigma \otimes \cO_{\bP^n}(jN) \lsm 
\sigma \otimes \cO_{\bP^n}(jN - N)[1] \lsm \dots \lsm \sigma[j].
\end{align}

We apply \cite[Corollary 2.2.2]{Polishchuk:families-of-t-structures} to $f=\iota: X \hookrightarrow \bP^n$. To check that $\pb \iota \sigma$ is a stability condition on $X$, we need to verify that for any $E \in \cP_\sigma(\theta)$ with $\theta \in (0,1]$, we have
\[
\iota_* \cO_X \otimes E \in \cP_\sigma(\le \theta).
\]

From the resolution \eqref{eq:resOX}, the object $\iota_* \cO_X$ admits a filtration in $\Db(\bP^n)$ with factors
\[
\cO_{\bP^n}, \quad \cO_{\bP^n}(-a_j)[j], \quad \cF[n].
\]
Tensoring with $E$, we get a filtration of $\iota_* \cO_X \otimes E$ with factors
\[
E, \quad E \otimes \cO_{\bP^n}(-a_j)[j], \quad E \otimes \cF[n].
\]

Hence
\begin{align}\label{eq:phi-min}
\phi^-_\sigma(\iota_* \cO_X \otimes E) \ge 
\min \big\{ \phi_\sigma(E), \phi^-_\sigma(E \otimes \cO_{\bP^n}(-a_j)[j]), \phi^-_\sigma(E \otimes \cF[n]) \big\}.
\end{align}

By \eqref{eq:twistineq} and Lemma~\ref{lem:lsm1}, for $1 \le j \le n-1$ we have
\begin{align}\label{eq:phi-twist}
\phi^-_\sigma(E \otimes \cO_{\bP^n}(-a_j)[j]) \ge \phi_\sigma(E).
\end{align}

Moreover, since all skyscraper sheaves are in $\cP_\sigma(1)$, \cite[Lemma 10.1]{Bridgeland:K3} implies
\[
E \in \cP_\sigma(\theta) \subset \Coh(\bP^n)[0,n-1], \quad \Coh(\bP^n) \subset \cP_\sigma((1-n,1]).
\]
As $\cF\in\Coh(\bP^n)$, it follows that
\[
E \otimes \cF[n] \in \Coh(\bP^n)[n, 3n-1] \subset \cP_\sigma((1,3n]),
\]
so in particular
\begin{align}\label{eq:phi-F}
\phi^-_\sigma(E \otimes \cF[n]) > 1 > \theta.
\end{align}

Combining \eqref{eq:phi-min}, \eqref{eq:phi-twist}, and \eqref{eq:phi-F}, we obtain
\[
\phi^-_\sigma(\iota_* \cO_X \otimes E) \ge \theta.
\]

Hence, the condition in \cite[Corollary 2.2.2]{Polishchuk:families-of-t-structures} is satisfied, and we conclude that
$\pb \iota \sigma$
defines a Bridgeland stability condition on $X$.
\end{proof}

\bibliography{rstab}                      % .bib-Datei

\newcommand{\etalchar}[1]{$^{#1}$}
\begin{thebibliography}{ATJLSS03}

\bibitem[ATJLSS03]{AJS:tstructure}
Leovigildo Alonso~Tarr\'io, Ana Jerem\'ias~L\'opez, and Mar\'ia~Jos\'e Souto~Salorio.
\newblock Construction of {$t$}-structures and equivalences of derived categories.
\newblock {\em Trans. Amer. Math. Soc.}, 355(6):2523--2543, 2003.

\bibitem[Bay19]{Arend:shortproof}
Arend Bayer.
\newblock A short proof of the deformation property of {B}ridgeland stability conditions.
\newblock {\em Math. Ann.}, 375(3-4):1597--1613, 2019.

\bibitem[Bri07]{Bridgeland:Stab}
Tom Bridgeland.
\newblock Stability conditions on triangulated categories.
\newblock {\em Ann. of Math. (2)}, 166(2):317--345, 2007, arXiv:math/0212237.

\bibitem[Bri08]{Bridgeland:K3}
Tom Bridgeland.
\newblock Stability conditions on {$K3$} surfaces.
\newblock {\em Duke Math. J.}, 141(2):241--291, 2008, arXiv:math/0307164.

\bibitem[FLZ22]{FLZ:ab3}
Lie Fu, Chunyi Li, and Xiaolei Zhao.
\newblock Stability manifolds of varieties with finite {A}lbanese morphisms.
\newblock {\em Trans. Amer. Math. Soc.}, 375(8):5669--5690, 2022.

\bibitem[KS08]{Kontsevich-Soibelman:stability}
Maxim Kontsevich and Yan Soibelman.
\newblock Stability structures, motivic {D}onaldson-{T}homas invariants and cluster transformations, 2008, arXiv:0811.2435.

\bibitem[Li25]{realred}
Chunyi Li.
\newblock A real reduction of the manifold of bridgeland stability conditions, 2025, arXiv:2506.21995.

\bibitem[Liu21]{yucheng:stabprodvar}
Yucheng Liu.
\newblock Stability conditions on product varieties.
\newblock {\em Journal für die reine und angewandte Mathematik (Crelles Journal)}, 2021(770):135--157, 2021.

\bibitem[LMP{\etalchar{+}}25]{stabhk}
Chunyi Li, Emanuele Macrì, Alexander Perry, Paolo Stellari, and Xiaolei Zhao.
\newblock Stability conditions on products of curves and hilbert schemes of surfaces, 2025, arXiv:2512.14207.

\bibitem[Mac91]{Macdonald1991}
I.~G. Macdonald.
\newblock {\em Notes on Schubert Polynomials}, volume~6.
\newblock Publications du LaCIM, Laboratoire de combinatoire et d’informatique math\'ematique (LaCIM), Universit\'e du Qu\'ebec \`a Montr\'eal, 1991.

\bibitem[Pol07]{Polishchuk:families-of-t-structures}
A.~Polishchuk.
\newblock Constant families of {$t$}-structures on derived categories of coherent sheaves.
\newblock {\em Mosc. Math. J.}, 7(1):109--134, 167, 2007, arXiv:math/0606013.

\bibitem[Soe07]{Soergel:bimodule}
Wolfgang Soergel.
\newblock Kazhdan-{L}usztig-{P}olynome und unzerlegbare {B}imoduln \"uber {P}olynomringen.
\newblock {\em J. Inst. Math. Jussieu}, 6(3):501--525, 2007.

\end{thebibliography}
\bibliographystyle{halpha}  
\end{document}